\documentclass[11pt,dvips,twoside,letterpaper]{article}
\usepackage{pslatex}
\usepackage{fancyhdr}
\usepackage{graphicx}
\usepackage{geometry}

\def\figurename{Figure} 
\makeatletter
\renewcommand{\fnum@figure}[1]{\figurename~\thefigure.}
\makeatother

\def\tablename{Table} 
\makeatletter
\renewcommand{\fnum@table}[1]{\centering\bfseries{\tablename~\thetable.}}
\makeatother

\usepackage{amsmath,amsthm,amsfonts}
\newtheorem{thm}{Theorem}[section]

\newtheorem{lem}[thm]{Lemma}

\theoremstyle{definition}

\theoremstyle{remark}
\newtheorem{rem}[thm]{Remark}
\numberwithin{equation}{section}


\begin{document}
\title{\bfseries \scshape{On a linear non-homogeneous  ordinary differential equation of the higher
order whose coefficients are  real-valued simple step functions}}
\author{\bfseries\itshape Gogi Pantsulaia\thanks{E-mail address: g.pantsulaia@gtu.ge},~Khatuna Chargazia \thanks{E-mail address: khatuna.chargazia@gmail.com},
Givi Giorgadze\thanks{E-mail address: g.giorgadze@gtu.ge}\\
I. Vekua Institute of Applied Mathematics, Tbilisi State
University,  B. P.
0143,\\ University St. 2, Tbilisi, Georgia \\
Department of Mathematics,~ Georgian Technical University, B. P.
0175,  \\ Kostava St. 77, Tbilisi 75, Georgia }

\date{}
\maketitle \thispagestyle{empty} \setcounter{page}{1}


\begin{abstract}By using the method developed in the paper [G.Pantsulaia, G.Giorgadze, On some applications of infinite-dimensional
cellular matrices, {\it Georg. Inter. J. Sci. Tech., Nova Science
Publishers,}  Volume 3,  Issue 1 (2011), 107-129], it is obtained
a representation  in  an explicit form of the particular solution
of the linear non-homogeneous  ordinary differential equation of
the higher order whose coefficients are
 real-valued  simple functions.
\end{abstract}

%
%
%


\vspace{.08in} \noindent \textbf{Key words and phrases:}
 linear ordinary differential equation, non-homogeneous ordinary differential equation.

\section{Introduction}

In \cite{Pan-Gio11} has been obtained a representation  in  an
explicit form of the particular solution of the linear
non-homogeneous ordinary differential equation of the higher order
with  real-valued coefficients. The aim of the present manuscript
is resolve an analogous problem for a linear non-homogeneous
ordinary differential equation of the higher order when
coefficients are real-valued simple step functions.

The paper is organized as follows.

In Section 2, we consider some auxiliary results obtained in the
paper \cite{Pan-Gio11}.  In Section 3, it is obtained  a
representation in an explicit form of the particular solution of
the linear non-homogeneous  ordinary differential equation of the
higher order whose coefficients are
 real-valued  simple functions. In Section 4 we present mathematical programm in MathLab for the
graphical solution of the corresponding differential equation.

\section{Some auxiliary propositions}

\medskip

 For $n \in N$, we denote by  $FD^{(n)}[-l,l[$  a vector
space of all $n$-times  differentiable functions $\Psi$  on
$[-l,l[$ such that a series obtained by $k$-times differentiation
term by term of the Fourier trigonometric series of $\Psi$
pointwise converges to $\Psi^{(k)}$ for all $x \in [-l,l[$ and $0
\le k \le n$.

Let  $(A_n)_{0 \le n \le 2M}$ be a sequence of real numbers, where
$M$ is any natural number. For each $k \ge 1$ we put
 $$\sigma_k=\sum_{n=0}^m (-1)^n A_{2n}(\frac{k \pi}{l})^{2n}, \eqno(2.1)$$
 $$ \omega_k=\sum_{n=0}^{m-1} (-1)^{n}A_{2n+1}(\frac{k\pi}{l})^{2n+1}.\eqno(2.2)$$

\begin{thm} (\cite{Pan-Gio11}, Theorem 3.1, p.45) For $m \ge 1$, let us consider an
ordinary differential equation
$$
\sum_{n=0}^{2m} A_n \frac{d^{n}}{d x^{n}} \Psi=f, \eqno(2.3)
$$
 where
$$
f(x)=\frac{c_0}{2}+\sum_{k=1}^{\infty}c_k \cos\Big(\frac{k \pi x
}{l}\Big)+ d_k \sin\Big(\frac{k \pi x }{l}\Big)\in FD^{(0)}[-l,l[
\eqno(2.4)
$$
and  $A_n \in R$ for $0 \le n \le 2m$.

Suppose that $A_0 \neq 0$ and $\sigma_k^2+\omega_k^2 \neq 0 $ for
$k \ge 1$,  where $\sigma_k$ and $\omega_k$ are defined by (2.1)
and (2.2), respectively.

If $(\frac{c_0}{2}, c_1, d_1, c_2, d_2, \dots )$ is such a
sequence of real numbers that the series $\Psi_p$, defined by
$$\Psi_p(x)=\frac{c_0}{2A_0} +\sum_{k=1}^{\infty}\Big( \frac{c_k\sigma_k-d_k\omega_k}{\sigma_k^2+\omega_k^2} \Big)     \cos\Big(\frac{k
\pi x }{l}\Big)+ \Big(
\frac{c_k\omega_k+d_k\sigma_k}{\sigma_k^2+\omega_k^2} \Big)
\sin\Big(\frac{k \pi x }{l}\Big), \eqno(2.5)
$$
belongs to the class $FD^{(2m)}[-l,l[$, then $\Psi_p$ is a
particular solution of  (2.3).
\end{thm}

\begin{thm}(\cite{Pan-Gio11}, Theorem 3.2, p.45) For $m \ge 1$, let us consider an ordinary differential equation
(2.3),
 where
$$
f(x) \in C[-l,l] \eqno(2.6)
$$
and  $A_n \in R$ for $0 \le n \le 2m$.

Suppose that $A_0 \neq 0$ and $\sigma_k^2+\omega_k^2 \neq 0 $ for
$k \ge 1$, where $\sigma_k$ and $\omega_k$ are defined by (2.1)
and (2.2), respectively. Let $(\frac{c_0}{2}, c_1, d_1, c_2, d_2,
\dots )$ be Fourier coefficients of $f$ and $(\frac{c_0}{2}, c_1,
d_1, c_2, d_2, \dots ) \in \ell_1.$

Then  the series $\Psi_p$,
defined by
$$\Psi_p(x)=\frac{c_0}{2A_0} +\sum_{k=1}^{\infty}\Big( \frac{c_k\sigma_k-d_k\omega_k}{\sigma_k^2+\omega_k^2} \Big)     \cos\Big(\frac{k
\pi x }{l}\Big)+ \Big(
\frac{c_k\omega_k+d_k\sigma_k}{\sigma_k^2+\omega_k^2} \Big)
\sin\Big(\frac{k \pi x }{l}\Big), \eqno(2.7)
$$
is a particular  solution of  (2.3).
\end{thm}

\section{A non-homogeneous ordinary differential equation of higher order whose coefficients are continuous or real-valued step
functions }

Let  consider a partition of $[-l,l[$ defined by
$$[-l,l[=\cup_{s=0}^{S-1}[\frac{l(2s-S)}{S},\frac{l(2s+2-S)}{S}[ )$$
We define a differential  operator
$$
L(\Psi)=\sum_{n=0}^{2m} A_n(x) \frac{d^{n}}{d x^{n}}\Psi
$$
for $\Psi \in FD^{(2m)}[-l,l[$. Notice that $L(\Psi)$ can be
rewritten as follows
$$
L(\Psi)=\cup_{s=0}^{S-1}Ind_{[\frac{l(2s-S)}{S},\frac{l(2s+2-S)}{S}[}\big(\sum_{n=0}^{2m}
A_n(x) \frac{d^{n}}{d x^{n}}\big) \Psi
$$
for $\Psi \in FD^{(2m)}[-l,l[$, where $Ind$ denotes an indicator
function.

For each $S \in \mathbb{N}$ we define an operator $L_S$  by

$$
L_S(\Psi)=\sum_{s=0}^{S-1}Ind_{[\frac{l(2s-S)}{S},\frac{l(2s+2-S)}{S}[}\big(\sum_{n=0}^{2m}
A_n(\frac{l(2s+1-S)}{S}) \frac{d^{n}}{d x^{n}}\big) \Psi
$$
for $\Psi \in FD^{(2m)}[-l,l]$.

\begin{lem} For each $\Psi \in FD^{(2m)}[-l,l[$ we have
$$
L(\Psi)=\lim_{S \to \infty}L_S(\Psi).
$$
\end{lem}

\medskip

\begin{thm}

 For $m \ge 1$, let us consider an
ordinary differential equation
$$
\sum_{n=0}^{2m} A_n(x) \frac{d^{n}}{d x^{n}} \Psi=f, \eqno(3.1)
$$
 where
$$
f(x)=\frac{c_0}{2}+\sum_{k=1}^{\infty}c_k \cos\Big(\frac{k \pi x
}{l}\Big)+ d_k \sin\Big(\frac{k \pi x }{l}\Big)\in FD^{(0)}[-l,l[
\eqno(3.2)
$$
and  $A_n(x) \in C[-l,l]$ for $0 \le n \le 2m$.

Suppose that $A_0(x) =1$  and $\sigma_k^2(x)+\omega_k^2(x) \neq 0
$ for $ x \in [-l,l[$ and  $k \ge 1$,  where $\sigma_k(x)$ and
$\omega_k(x)$ are defined by
$$\sigma_k(x)=\sum_{n=0}^m (-1)^n A_{2n}(x)(\frac{k \pi}{l})^{2n}, \eqno(3.3)$$
 $$ \omega_k(x)=\sum_{n=0}^{m-1} (-1)^{n}A_{2n+1}(x)(\frac{k\pi}{l})^{2n+1}.\eqno(3.4)$$

Suppose that the following conditions are valid:

(i) $(\frac{c_0}{2}, c_1, d_1, c_2, d_2, \dots ) \in \ell_1$;

(iii) There is a  constant $C>0$ such that

 $$
 \Big|\frac{\omega_k(x+h)}{\sigma_k^2(x+h)+\omega_k^2(x+h)}-
 \frac{\omega_k(x)}{\sigma_k^2(x)+\omega_k^2(x)}\Big|\le
 C|h|^{2}
 $$
and
$$
 \Big|\frac{\sigma_k(x+h)}{\sigma_k^2(x+h)+\omega_k^2(x+h)}-\frac{\sigma_k(x)}{\sigma_k^2(x)+\omega_k^2(x)}\Big|\le
 C|h|^{2}.
 $$

Then the function  $\Psi_0$, defined by
$$\Psi_0(x)=\frac{c_0}{2} +\sum_{k=1}^{\infty}\Big( \frac{c_k\sigma_k(x)-d_k\omega_k(x)}
{\sigma_k^2(x)+\omega_k^2(x)} \Big)     \cos\Big(\frac{k \pi x
}{l}\Big)+ $$
 $$\Big( \frac{c_k\omega_k(x)+d_k\sigma_k(x)}{\sigma_k^2(x)+\omega_k^2(x)} \Big)
  \sin\Big(\frac{k \pi x }{l}\Big), \eqno(3.5)
$$
 is a
particular solution of  (3.1).
\end{thm}

\begin{proof}
We put
$$
\Psi_S(x)=\sum_{s=0}^{S-1}Ind_{[\frac{l(2s-S)}{S},\frac{l(2s+2-S)}{S}[}(x)\Big[
\frac{c_0}{2} +
$$
$$\sum_{k=1}^{\infty}\Big(
\frac{c_k\sigma_k(\frac{l(2s+1-S)}{S})-d_k\omega_k(\frac{l(2s+1-S)}{S})}
{\sigma_k^2(\frac{l(2s+1-S)}{S})+ \omega_k^2(\frac{l(2s+1-S)}{S})}
\Big) \cos\Big(\frac{k \pi x }{l}\Big)+
 $$
 $$
 \Big( \frac{c_k\omega_k(\frac{l(2s+1-S)}{S})+d_k\sigma_k(\frac{l(2s+1-S)}{S})}
 {\sigma_k^2(\frac{l(2s+1-S)}{S})+ \omega_k^2(\frac{l(2s+1-S)}{S})} \Big)
  \sin\Big(\frac{k \pi x }{l}\Big) \Big], \eqno(3.6)
$$
On the one hand, by using the result of Lemma 3.1 we have
$$
L(\lim_{S \to \infty}\Psi_S(x))= \lim_{S \to \infty}L(\Psi_S(x))=
\lim_{S \to \infty}L_S(\Psi_S(x))=\lim_{S \to \infty}f(x)=f(x).
$$

On the other hand we have
$$
|\lim_{S \to \infty}\Psi_S(x)-\Psi_0(x))|=\lim_{S \to
\infty}|\Psi_S(x)-\Psi_0(x)|=\lim_{S \to \infty}\Big|
\sum_{s=0}^{S-1}Ind_{[\frac{l(2s-S)}{S},\frac{l(2s+2-S)}{S}[}(x)\Big[
\Big(\frac{c_0}{2}-\frac{c_0}{2}\Big) +
$$
$$\sum_{k=1}^{\infty}\Big(
\frac{c_k\sigma_k(\frac{l(2s+1-S)}{S})-d_k\omega_k(\frac{l(2s+1-S)}{S})}
{\sigma_k^2(\frac{l(2s+1-S)}{S})+
\omega_k^2(\frac{l(2s+1-S)}{S})}-\frac{c_k\sigma_k(x)-d_k\omega_k(x)}
{\sigma_k^2(x)+ \omega_k^2(x)} \Big) \cos\Big(\frac{k \pi x
}{l}\Big)+
 $$
 $$
 \Big( \frac{c_k\omega_k(\frac{l(2s+1-S)}{S})+d_k\sigma_k(\frac{l(2s+1-S)}{S})}
 {\sigma_k^2(\frac{l(2s+1-S)}{S})+ \omega_k^2(\frac{l(2s+1-S)}{S})}
- \frac{c_k\omega_k(x)+d_k\sigma_k(x)}
 {\sigma_k^2(x)+ \omega_k^2(x)}
 \Big)
  \sin\Big(\frac{k \pi x }{l}\Big) \Big]  \Big|\le
  $$
$$
\lim_{S \to \infty} \sum_{s=0}^{S-1}\sup_{x \in
[\frac{l(2s-S)}{S},\frac{l(2s+2-S)}{S}[}\{
\sum_{k=1}^{\infty}\Big|
\frac{c_k\sigma_k(\frac{l(2s+1-S)}{S})-d_k\omega_k(\frac{l(2s+1-S)}{S})}
{\sigma_k^2(\frac{l(2s+1-S)}{S})+
\omega_k^2(\frac{l(2s+1-S)}{S})}-\frac{c_k\sigma_k(x)-d_k\omega_k(x)}
{\sigma_k^2(x)+ \omega_k^2(x)} \Big|+
 $$
 $$
 \Big | \frac{c_k\omega_k(\frac{l(2s+1-S)}{S})+d_k\sigma_k(\frac{l(2s+1-S)}{S})}
 {\sigma_k^2(\frac{l(2s+1-S)}{S})+ \omega_k^2(\frac{l(2s+1-S)}{S})}
- \frac{c_k\omega_k(x)+d_k\sigma_k(x)}
 {\sigma_k^2(x)+ \omega_k^2(x)}
 \Big|\}=
  $$
$$
\lim_{S \to \infty} \sum_{s=0}^{S-1}\sup_{x \in
[\frac{l(2s-S)}{S},\frac{l(2s+2-S)}{S}[}\{
\sum_{k=1}^{\infty}\Big| c_k(\frac{\sigma_k(\frac{l(2s+1-S)}{S})}
{\sigma_k^2(\frac{l(2s+1-S)}{S})+
\omega_k^2(\frac{l(2s+1-S)}{S})}-\frac{\sigma_k(x)}
{\sigma_k^2(x)+ \omega_k^2(x)})-
 $$
 $$
d_k(\frac{\omega_k(\frac{l(2s+1-S)}{S}))}
{\sigma_k^2(\frac{l(2s+1-S)}{S}) +
\omega_k^2(\frac{l(2s+1-S)}{S})}-\frac{\omega_k(x)}
{\sigma_k^2(x)+ \omega_k^2(x)})\Big|+
$$

$$\Big|
d_k(\frac{\sigma_k(\frac{l(2s+1-S)}{S})}
{\sigma_k^2(\frac{l(2s+1-S)}{S}) +
\omega_k^2(\frac{l(2s+1-S)}{S})}-\frac{\sigma_k(x)}
{\sigma_k^2(x)+ \omega_k^2(x)})+
 $$
 $$
c_k(\frac{\omega_k(\frac{l(2s+1-S)}{S}))}
{\sigma_k^2(\frac{l(2s+1-S)}{S}) +
\omega_k^2(\frac{l(2s+1-S)}{S})}-\frac{\omega_k(x)}
{\sigma_k^2(x)+ \omega_k^2(x)})\Big|\}\le
  $$

$$
\lim_{S \to \infty} \sum_{s=0}^{S-1}\sup_{x \in
[\frac{l(2s-S)}{S},\frac{l(2s+2-S)}{S}[}\{
\sum_{k=1}^{\infty}2(|c_k|+|d_k|)\frac{4l^2C}{S^2}\}\le
$$

$$
\lim_{S \to
\infty}\frac{8l^2C}{S}\sum_{k=1}^{\infty}(|c_k|+|d_k|)=0.
$$

\end{proof}

\begin{rem} Theorem 3.2 is a generalization of Theorem 2.2. Indeed, Theorem 2.2 is a simple consequence of Theorem 3.2,
when $A_n(x)=const$ for $0 \le n \le 2m$, because in that cases
all conditions of Theorem 3.2 are fulfilled.
\end{rem}

We say that  $(a_k)_{0 \le k \le s}$ is partition of $[-l,l[$ if
$-l=a_0< a_1 < \cdots < a_{s-1}<a_s=l$.

We say that a real-valued function $f$  on $[-l,l[$ is simple
function if there exists a partition  $(a_k)_{0 \le k \le s}$ of
$[-l,l[$ and a sequence of real numbers $(A_k)_{1 \le k \le s}$
such that $$f(x)=\sum_{k=1}^s A_k Ind_{[a_{k-1},a_k)}(x)$$ for $x
\in [-l,l[$.

We have the following proposition.
\begin{thm}
Suppose that $(A_n(x))_{0 \le n \le 2m}$ is a sequence of
real-valued simple step functions  on $[-l,l[$, i.e. for every
$n~(0 \le n \le 2m)$  there exists a partition $(a^{(n)}_k)_{0 \le
k \le s_n}$ of $[-l,l[$ and a sequence of real numbers
$(A^{(n)}_k)_{1 \le k \le s_n}$ such that
$$A_n(x)=\sum_{k=1}^{s_n}A^{(n)}_k Ind_{[a^{(n)}_{k-1},a^{(n)}_k)}(x)$$ for
$x \in [-l,l[$.

Suppose that $A_0(x)$ does not remain a zero value on $[-l,l[$ and
$\sigma_k^2(x)+\omega_k^2(x) \neq 0 $ for $ x \in [-l,l[$ and  $k
\ge 1$,  where $\sigma_k(x)$ and $\omega_k(x)$ are defined by
(3.3) and (3.4). Suppose also that Fourier coefficients of the
function $f$ standing in the right side of the equation (3.1)
satisfy the following condition $(\frac{c_0}{2}, c_1, d_1, c_2,
d_2, \dots ) \in \ell_1$.

Then the function  $\Psi_0$, defined by
$$\Psi_0(x)=\frac{c_0}{2A_0(x)} +\sum_{k=1}^{\infty}\Big( \frac{c_k\sigma_k(x)-d_k\omega_k(x)}
{\sigma_k^2(x)+\omega_k^2(x)} \Big)     \cos\Big(\frac{k \pi x
}{l}\Big)+ $$
 $$\Big( \frac{c_k\omega_k(x)+d_k\sigma_k(x)}{\sigma_k^2(x)+\omega_k^2(x)} \Big)
  \sin\Big(\frac{k \pi x }{l}\Big), \eqno(3.7)
$$
for $x \in [-l,l[$,  satisfies (3.1)--(3.2) at each point of the
set
 $$(-l,l) \setminus
\cup_{0 \le n \le 2m}\{a^{(n)}_1, a^{(n)}_2, \cdots,
a^{(n)}_{s_n-1}\}$$  .
\end{thm}

\begin{proof}
If $x_0 \in (-l,l) \setminus G (:=\cup_{0 \le n \le
2m}\{a^{(n)}_1, a^{(n)}_2, \cdots, a^{(n)}_{s_n-1}\})$, then by
virtue of the openness of the $G$  there exists a positive real
number $r>0$ such that $(x_0-r,x_0+r) \subseteq G$.  It is obvious
that $(A_n(x)$ is constant on  $(x_0-r,x_0+r)$ for   $0 \le n \le
2m$. We set $A_n:=A_n(x_0)$ for $0 \le n \le 2m$.

For $m \ge 1$, let us consider an ordinary differential equation
$$
\sum_{n=0}^{2m} A_n \frac{d^{n}}{d x^{n}} \Psi=f. \eqno(3.8)
$$

Note that for (3.8) all conditions of Theorem 2.2 are fulfilled.
Hence the series $\Psi_p$, defined by
$$\Psi_p(x)=\frac{c_0}{2A_0} +\sum_{k=1}^{\infty}\Big( \frac{c_k\sigma_k-d_k\omega_k}{\sigma_k^2+\omega_k^2} \Big)     \cos\Big(\frac{k
\pi x }{l}\Big)+ $$
 $$\Big( \frac{c_k\omega_k+d_k\sigma_k}{\sigma_k^2+\omega_k^2} \Big)   \sin\Big(\frac{k \pi x }{l}\Big), \eqno(3.9)
$$
is a particular  solution of  (3.8), where $\sigma_k$ and
$\omega_k$ are defined by (2.1) and (2.2), respectively.

Notice that $\Psi_p$ defined by (3.9) coincides with $\psi_0$
defined by (3.7) at all point $x \in (x_0-r,x_0+r)$. Similarly,
the equation (3.8) with (3.2) coincides with the equation (3.1)
with (3.2)  at all point $x \in (x_0-r,x_0+r)$. Hence $\psi_0$
defined by (3.7) satisfies (3.1)--(3.2) at each point of the set
$(x_0-r,x_0+r)$, in particular, at point $x_0$. Since $x_0 \in
(-l,l) \setminus G$ was taken arbitrary, we  end the proof of
Theorem 3.4.

\end{proof}

\section{ On a graphical solution of the linear non-homogeneous ordinary differential
equation of the higher order  whose coefficients are real-valued
simple step functions}

Let consider the linear non-homogeneous ordinary differential
equation  of the $22$-th order
$$
\Psi(x)+A_2(x)\frac{d^2}{dx^2}\Psi(x)+A_5(x)\frac{d^5}{dx^5}\Psi(x)+A_{14}(x)\frac{d^{14}}{dx^{14}}\Psi(x)+
A_{20}(x)\frac{d^{20}}{dx^{20}}\Psi(x)+
$$
$$A_{22}(x)\frac{d^{22}}{dx^{22}}\Psi(x) =1+2 \cos(x),\eqno(4.1)
$$
where

$$A_2(x)=-0.001 \times Ind_{[-\pi,-\pi/2[}(x)-0.002 \times
Ind_{[-\pi/2,0[}(x)-0.001\times Ind_{[0,\pi/2[}(x)-0.002\times
Ind_{[\pi/2, \pi[}(x), $$

$$A_5(x)=0.01 \times Ind_{[-\pi,-\pi/2[}(x)-0.01 \times
Ind_{[-\pi/2,0[}(x)+0.002\times Ind_{[0,\pi/2[}(x)-0.002\times
Ind_{[\pi/2, \pi[}(x), $$

$$A_{14}(x)=0.1 \times Ind_{[-\pi,-\pi/2[}(x)-0.1 \times
Ind_{[-\pi/2,0[}(x)-0.4\times Ind_{[0,\pi/2[}(x)+0.007\times
Ind_{[\pi/2, \pi[}(x), $$

$$A_{20}(x)=-0.01 \times Ind_{[-\pi,-\pi/2[}(x)+0.01 \times
Ind_{[-\pi/2,0[}(x)+0.002\times Ind_{[0,\pi/2[}(x)-0.22\times
Ind_{[\pi/2, \pi[}(x), $$

$$A_{22}(x)=0.001 \times Ind_{[-\pi,-\pi/2[}(x)-0.001 \times
Ind_{[-\pi/2,0[}(x)+0.0003\times Ind_{[0,\pi/2[}(x)-0.0003\times
Ind_{[\pi/2, \pi[}(x). $$

\medskip

{\bf Definition 4.1} We say that $g \in FD^{(22)}([-\pi,\pi]
\setminus G)~(G:=\{ -\pi, -\pi/2, 0,\pi/2, \pi\})$ if
$$g(x)=g_1(x)  \times Ind_{[-\pi,-\pi/2[}(x)+g_2(x)  \times
Ind_{[-\pi/2,0[}(x)+g_3(x) \times Ind_{[0,\pi/2[}(x)+g_4(x) \times
Ind_{[\pi/2, \pi[}(x) \eqno(4.2)$$
 for
some $g_1,g_2,g_3,g_4 \in FD^{(22)}([-\pi,\pi[)$.

\medskip

Below we present the program in MathLab which gives the graphical
solution of the differential equation (4.1) in the class
$FD^{(22)}([-\pi,\pi[ \setminus G)$.

{\bf

$A1=[0, -0.001,0, 0, 0.01,   0,0,0,0,0,  0,0,0,0.1,0,
0,0,0,0,-0.01,  0,0.001];$

$A2=[0, -0.002,0, 0, -0.01,   0,0,0,0,0,  0,0,0,-0.1,0,
0,0,0,0,0.01,  0,-0.001];$

$A3=[0, -0.001,0, 0, 0.002,   0,0,0,0,0,  0,0,0,-0.4,0,
0,0,0,0,0.002,  0,0.0003];$

$A4=[0, -0.002,0, 0, -0.002,   0,0,0,0,0,  0,0,0,0.007,0,
0,0,0,0,-0.22,  0,-0.0003];$

$C=[2, 0,0, 0, 0,0,0,0,0,0,0,0,0,0,0,0,0,0,0,0,0,0];$

$D=[0,0,0, 0, 0,0,0,0,0,0,0,0,0,0,0,0,0,0,0,0,0,0];$

$C0=2; A10=1; A20=1; A30=1; A40=1;$

$x=1:20;$

$S1=A10;~ S2=A20;~  S3=A30;~  S4=A40;$

for $k=1:11$

$S1=S1+(-1)^{}(k)*A1(2*k)*x.^{}(2*k);$

$S2=S2+(-1)^{}(k)*A2(2*k)*x.^{}(2*k);$

$S3=S3+(-1)^{}(k)*A3(2*k)*x.^{}(2*k);$

$S4=S4+(-1)^{}(k)*A4(2*k)*x.^(2*k);$

end

$O1=A1(1); O2=A2(1); O3=A3(1); O4=A4(1);$

for $k=1:10$

$O1=O1 +(-1)^{}k*A1(2*k+1)*x.^{}(2*k+1);$

$O2=O2 +(-1)^{}k*A2(2*k+1)*x.^{}(2*k+1);$

$O3=O3 +(-1)^{}k*A3(2*k+1)*x.^{}(2*k+1);$

$O4=O4 +(-1)^{}k*A4(2*k+1)*x.^{}(2*k+1);$

end

$x1=(-pi): ( pi/100) : (-pi/2);$

$y1=  C0 /(2*A10);$

 for $n=1:20$

$y1=y1+cos(n*x1).*(C(n)*S1(n)-D(n)*O1(n))/(S1(n)^{}2+O1(n)^{}2)+$

$sin(n* x1).*(C(n)*O1(n)+D(n)*S1 (n))/(S1(n)^{}2+O1(n)^{}2); $

end

$x2=(-pi/2): ( pi/100) :0;$

$y2=  C0 /(2*A20);$

for $n=1:20$

$y2=y2+cos(n*x2).*(C(n)*S2(n)-D(n)*O2(n))/(S2(n)^{}2+O2(n)^{}2)+$

$sin(n* x2).*(C(n)*O2(n)+ D(n)*S2 (n))/(S2(n)^{}2+O2(n)^{}2); $

end

$x3=0: ( pi/100) : (pi/2);$

$y3=  C0 /(2*A30);$

for $n=1:20$

$y3=y3+cos(n*x3).*(C(n)*S3(n)-D(n)*O3(n))/(S3(n)^{}2+O3(n)^{}2)+$

$sin(n* x3).*(C(n)*O3(n)+ D(n)*S3 (n))/(S3(n)^{}2+O3(n)^{}2); $

end

$x4=(pi/2): ( pi/100) :pi;$

$y4=  C0 /(2*A40);$

for $n=1:20$

$y4=y4+cos(n*x4).*(C(n)*S4(n)-D(n)*O4(n))/(S4(n)^{}2+O4(n)^{}2)+$

$sin(n* x4).*(C(n)*O4(n)+ $ $D(n)*S4 (n))/(S4(n)^{}2+O4(n)^{}2);$

end

for $i=1:20$

if $O1(i)^{}2+S1(i)^{}2~=0; O2(i)^{}2+S2(i)^{}2~=0;
O3(i)^{}2+S3(i)^{}2~=0; O3(i)^{}2+S3(i)^{}2~=0;$

plot$(x1,y1,x2,y2, x3,y3,x4,y4)$

else error('the ordinary differential  equation  has no solution
or has infinitely many solutions' in the class
$FD^{(22)}([-\pi,\pi] \setminus G)$)

end

end }

\medskip

On Figure 1, the graphical solution of the differential equation
(4.1) is presented.

\begin{figure}[h]
\center{\includegraphics[width=1\linewidth]{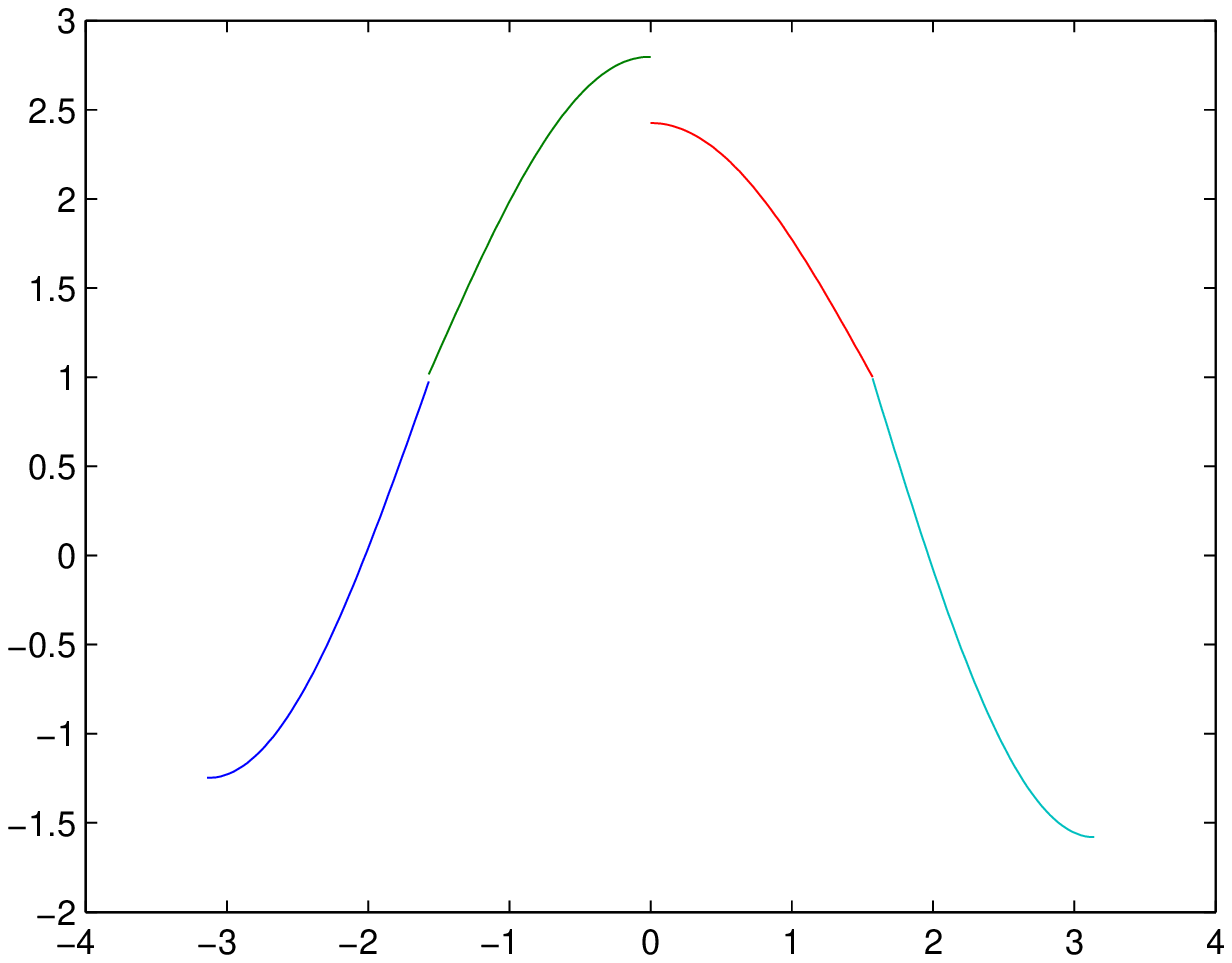}}
\caption{Graphical solution of the ODE (4.1).} \label{ris:image}
\end{figure}

{\bf Remark 4.1} Notice that for each natural number $M>1$, one
can easily modify this program in MathLab for obtaining a
graphical solution of the differential equation (3.1)-(3.2) in
$FD^{(2M)}([-l,l[\setminus G)$ whose coefficients $(A_n(x))_{0 \le
n \le 2M}$ are real-valued simple step functions on $[-l,l[$, $f$
is a trigonometric polynomial on $[-l,l[$ and  $G$ is the
partition of the interval $[-l,l[$ defined by the family
$(A_n(x))_{0 \le n \le 2M}$.

{\bf Remark 4.2} Since each  constant $c$ admits the following
evident representation
$$c=c \times Ind_{[-\pi,-\pi/2[}(x)+c  \times
Ind_{[-\pi/2,0[}(x)+c \times Ind_{[0,\pi/2[}(x)+c \times
Ind_{[\pi/2, \pi[}(x), \eqno(4.3)$$ we can use above mentioned
program for a solution of the differential equation (2.3)-(2.4)
with constant coefficients.

\begin{figure}[h]
\center{\includegraphics[width=1\linewidth]{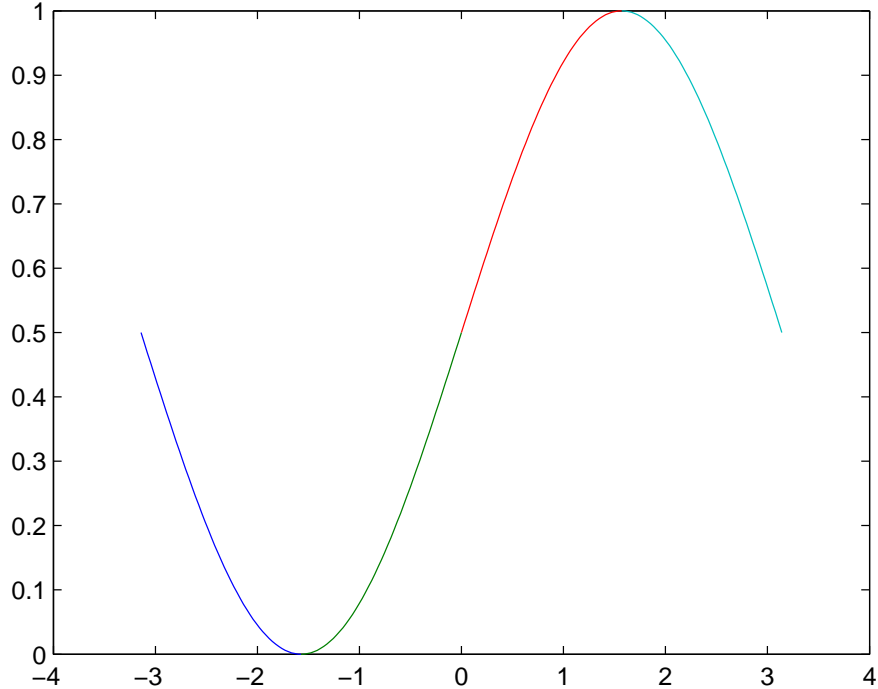}}
\caption{Graphical solution of the ODE (4.4).} \label{ris:image}
\end{figure}

On Figure 2, the graphical solution of the  linear non-homogeneous
ordinary differential equation of the of the second order with
real-valued constant coefficients

$$
\Psi(x)-\frac{d^2}{dx^2}\Psi(x)=1/2+\cos(x), \eqno(4.4)
$$
is presented, which has been obtained by entering in the above
mentioned program  of the following data:

{\bf

$A1=[0, -1,0, 0, 0,0,0,0,0,0,0,0,0,0,0,0,0,0,0,0,0,0];$

$A2=[0, -1,0, 0, 0,0,0,0,0,0,0,0,0,0,0,0,0,0,0,0,0,0];$

$A3=[0, -1,0, 0, 0,0,0,0,0,0,0,0,0,0,0,0,0,0,0,0,0,0];$

$A4=[0, -1,0, 0, 0,0,0,0,0,0,0,0,0,0,0,0,0,0,0,0,0,0];$

$C=[0, 0,0, 0, 0,0,0,0,0,0,0,0,0,0,0,0,0,0,0,0,0,0];$

$D=[1,0,0, 0, 0,0,0,0,0,0,0,0,0,0,0,0,0,0,0,0,0,0];$

$C0=1;~ A10=1;~ A20=1;~ A30=1;~ A40=1;$}

\begin{rem} The approach of Theorem 3.4 used for a solution of
(3.1)-(3.2) with real-valued simple step functions $(A_n(x))_{0\le n
\le 2M}$ can be used in such a case when the corresponding coefficients
 are continuous functions on
$[-l,l[$. If we will approximate these coefficients  by real-valued simple step functions, then it is natural
to wait that under some "nice restrictions" on these coefficients  the solution obtained by Theorem 3.4, will be
a "good approximation" of the corresponding solution.
\end{rem}



\end{document}